\documentclass[review]{elsarticle}

\usepackage{lineno,hyperref}
\usepackage{mathrsfs}
\usepackage{amsfonts}
\usepackage{amssymb}
\usepackage{amsmath}
\usepackage{hyperref}
\usepackage{amsthm}
\usepackage{cases}
\usepackage{arydshln}
\usepackage[usenames,dvipsnames]{color}
\usepackage{graphicx}%
\newtheorem{thm}{Theorem}[section]

\newtheorem{lemma}[thm]{Lemma}

\newtheorem{proposition}[thm]{Proposition}

\newtheorem{conj}[thm]{Conjecture}

\modulolinenumbers[5]

\begin{document}

\begin{frontmatter}

\title{On Drury's solution of Bhatia \& Kittaneh's question\tnoteref{mytitlenote}}
\tnotetext[mytitlenote]{Dedicated to Rajendra Bhatia on the occassion of his 65th birthday.}

\author{Minghua Lin\fnref{myfootnote}}
\address{Department of Mathematics, Shanghai University, Shanghai, 200444, China}
\fntext[myfootnote]{Email: mlin87@ymail.com}

\begin{abstract} Let $A, B$ be $n\times n$ positive semidefinite matrices. Bhatia and Kittaneh  asked whether it is true $$ \sqrt{\sigma_j(AB)}\le \frac{1}{2} \lambda_j(A+B),   \qquad  j=1, \ldots, n$$ 
 where $\sigma_j(\cdot)$, $\lambda_j(\cdot)$, are the $j$-th largest singular value, eigenvalue, respectively. The question was recently solved by Drury in the affirmative.  This article revisits Drury's solution. In particular, we simplify the proof for a key auxiliary result  in his solution.  
 
\end{abstract}


\end{frontmatter}


\section{Introduction}
 Bhatia has made many fundamental contributions to Matrix Analysis  \cite{Bha97}. One of his favorite topics is matrix inequalities.  Roughly speaking, matrix inequalities are noncommutative versions of the corresponding scalar inequalities. To get a glimpse of this topic, let us start with a simple example. The simplest AM-GM inequality says that 
 
 $$a, b>0 \implies \frac{a+b}{2}\ge \sqrt{ab}.$$ 
 
Now it is known that \cite[p. 107]{Bha07} its most ``direct" noncommutative version is 
\begin{eqnarray}\label{am-gm}
 A, B ~ \hbox{are $n\times n$ positive definite matrices} \implies \frac{A+B}{2}\ge A\sharp B,
\end{eqnarray} 
where $A\sharp B:=A^{\frac{1}{2}}(A^{-\frac{1}{2}}BA^{-\frac{1}{2}})^{\frac{1}{2}}A^{\frac{1}{2}}$ is called the geometric mean of $A$ and $B$. For two Hermitian matrices $A$ and $B$ of the same size, in this article, we write $A\ge B$ (or $B\le A$) to mean that $A-B$ is positive semidefinite. 

If we denote $S:=A\sharp B$, then $B=SA^{-1}S$. Thus a variant of (\ref{am-gm}) is the following
\begin{eqnarray}\label{am-gm1}
A, S ~ \hbox{are $n\times n$ positive definite matrices} \implies  A+SA^{-1}S\ge 2S.
\end{eqnarray}

There is a long tradition in matrix analysis of comparing eigenvalues or singular values. To proceed, let us fix some notation. The $j$-th largest singular value of  a complex matrix  $A$ is denoted by $\sigma_j(A)$. If all the eigenvalues of $A$ are real, then we denote its $j$-th largest one by $\lambda_j(A)$. By  Weyl's Monotonicity Theorem  \cite[p. 63]{Bha97},  (\ref{am-gm}) readily implies 
\begin{eqnarray*}   \lambda_j(A+B)\ge  2 \lambda_j(A\sharp B), \qquad  j=1, \ldots, n. 
\end{eqnarray*}

As far as the eigenvalues or singular values are considered, there are other versions of ``geometric mean".  Bhatia and Kittaneh studied this kind of inequalities over a twenty year period \cite{BK90, BK00, BK08}. Their elegant results include the following: If $A, B$ are $n\times n$ positive semidefinite matrices, then 
\begin{eqnarray}\label{bk1} && \lambda_j(A+B)\ge  2\sqrt{\lambda_j(AB)}=2\sigma_j(A^{\frac{1}{2}}B^{\frac{1}{2}}); \\&&
	\label{bk2}  \lambda_j(A+B)\ge   2\lambda_j(A^{\frac{1}{2}}B^{\frac{1}{2}}) 
\end{eqnarray} for  $j=1, \ldots, n$.

To complete the picture in (\ref{bk1})-(\ref{bk2}), they asked whether it is true 
\begin{eqnarray*}\lambda_j(A+B)\ge  2\sqrt{\sigma_j(AB)}, \qquad  j=1, \ldots, n?
\end{eqnarray*}  

This question was recently answered in the affirmative by Drury in his very brilliant work \cite{Dru12}. The purpose of this expository article is to revisit Drury's solution.   Hopefully, some of our arguments would shed new insights into the beautiful result, which is now a theorm. 

\begin{thm}\cite{Dru12}  If $A, B$ are $n\times n$ positive definite semidefinite matrices, then 
\begin{eqnarray}
 \label{bkd}   	\lambda_j(A+B)\ge  2\sqrt{\sigma_j(AB)}, \qquad  j=1, \ldots, n.
\end{eqnarray} 
  \end{thm}

 \section{Drury's reduction in proving (\ref{bkd})}
 Our presentation here is just slightly different from that in \cite{Dru12}.  
 
 Assume without loss of generality that $A, B$ are positive definite (the general case is by a standard purturbation argument). Fix $r$ in the range $1\le r\le n$ and normalize so that 
 $\sigma_r(AB)=1$. Our goal is to  show that  $\lambda_r(A+B)\ge 2$. 
 
 Note that $\sigma_r(AB)=1$ is the same as $\lambda_r(AB^2A)=1$. Consider the spectral decomposition $$AB^2A=\sum_{k=1}^n\lambda_k(AB^2A)P_k,$$ where $P_1, P_2, \ldots, P_n$, are orthogonal projections. Then
 $\lambda_k(AB^2A)\ge 1$ for $k=1, \ldots, r$. Define a positive semidefinite $$B_1:=\left(A^{-1}\left(\sum_{k=1}^r P_k\right)A^{-1}\right)^{1/2}.$$ It is easy to see (indeed, from $B^2\ge B_1^2$) that 
 $$B=\left(A^{-1}\left(\sum_{k=1}^r\lambda_k(AB^2A) P_k\right)A^{-1}\right)^{1/2}\ge B_1.$$ So we are done if we can show  \begin{eqnarray}\label{reduction1}
  \lambda_r(A+B_1)\ge 2. 
 \end{eqnarray}
 
 As $B_1$ has rank $r$, split the underlying space as the direct sum of image and kernel of $B_1$, we may partition comformally $B_1$ and $A$ in the following form
 $$B_1=\begin{pmatrix}
 X & 0 \\0& 0
 \end{pmatrix}, \  A=\begin{pmatrix}
 A_{11} & A_{12}\\ A_{12}^*&  A_{22}
 \end{pmatrix}.$$
 
 Note $AB_1^2A$ is an orthogonal projection of rank $r$, the same is true for $B_1A^2B_1$. Therefore, $$B_1A^2B_1=\begin{pmatrix}
 X(A_{11}^2+A_{12}A_{12}^*)X & 0\\ 0& 0
 \end{pmatrix}\implies X(A_{11}^2+A_{12}A_{12}^*)X=I_r$$
 where $I_r$ is the $r\times r$ identity matrix. 
 
 Finally, observe that $$A\ge A_1:=\begin{pmatrix}
 	A_{11} & A_{12}\\ A_{12}^*&   	A_{12}^*A_{11}^{-1}A_{12}
 \end{pmatrix}.$$
Therefore,  (\ref{reduction1}) would follow from 
  \begin{eqnarray}\label{reduction2}
  \lambda_r(A_1+B_1)\ge 2. 
  \end{eqnarray}
  
Thus, the remaining effort is made to show   (\ref{reduction2}), which we formulate as a proposition. 
\begin{proposition}\label{p1}
Let $A_{11}$ and $X$ be $r\times r$ positive definite matrices and $A_{12}$ is an $(n-r)\times (n-r)$ matrix such that $X(A_{11}^2+A_{12}A_{12}^*)X=I_r$. Then \begin{eqnarray}\label{reduction3}
 \lambda_r\begin{pmatrix}
 A_{11}+X & A_{12}\\ A_{12}^*&  A_{12}^*A_{11}^{-1}A_{12}  
 \end{pmatrix}\ge 2.
\end{eqnarray} 
\end{proposition}

 \section{The mystified part}
 In order to prove (\ref{reduction3}), Drury made the following key observations. 
 \begin{proposition}\label{p2}\cite[Proposition 2]{Dru12} Let $M$ and $N$ be $r\times r$ positive definite matrices. Then 
 \begin{eqnarray*}
 \lambda_r\begin{pmatrix}
M &(M\sharp N)^{-1}\\ (M\sharp N)^{-1}& N 
 \end{pmatrix}\ge 2.
 \end{eqnarray*} 	 \end{proposition}
 	\begin{proposition}\label{p3}\cite[Theorem 7]{Dru12} Let $L$ and $M$ be $r\times r$ positive definite matrices, and let $Z$ be an $r\times r$ matrix such that $ML(I+ZZ^*)LM=I_r$. Then 
 		\begin{eqnarray} \label{reduction4} \lambda_r\begin{pmatrix}
 				L+M & LZ\\ Z^*L&  Z^*LZ 
 			\end{pmatrix}\ge 2.
 		\end{eqnarray} 	 \end{proposition}

 The way that Drury proved (\ref{reduction4}) is by showing that $T:=\begin{pmatrix}
 L+M & LZ\\ Z^*L&  Z^*LZ 
 \end{pmatrix}$ and $R:=\begin{pmatrix}
 M &(M\sharp N)^{-1}\\ (M\sharp N)^{-1}& N 
 \end{pmatrix}$ have the same characteristic polynomial, and so the eigenvalues of $R$ and $T$ coincide.  As explained in \cite{Dru12a}, this connection (between $R$ and $T$) is mystified. Formally, the  mystified part also comes from $R$ and $T$ themselves, indeed, $T$ is always positive semidefinite while  $R$ is not!
 
 In order to apply Proposition \ref{p3} to Proposition \ref{p1}, Drury discussed three possible relations between the size $n$ and $r$. Our proof of Proposition \ref{p1} in the next section allows us to skip this discussion on the size. 
 
   \section{Proof of Proposition \ref{p1}}
   The following lemma slightly generalizes Proposition \ref{p2} in form. 
   \begin{lemma} \label{lem1} Let $X$ be a $r\times r$ positive definite matrix and let $S$ be a  $r\times r$ nonsingular matrix. Then 
   	\begin{eqnarray*}
   		\lambda_r\begin{pmatrix}
   			SX^{-1}S^*& (S^{-1})^*\\S^{-1}& X
   		\end{pmatrix}\ge 2.
   	\end{eqnarray*}    	 
   \end{lemma}
   \begin{proof}
   	 Consider the polar decomposition of $S$, $S=U|S|$, where $U$ is unitary and $|S|=(S^*S)^{\frac{1}{2}}$. The matrix $\begin{pmatrix}
   	 	SX^{-1}S^*& (S^{-1})^*\\S^{-1}& X
   	 \end{pmatrix}$ is unitarily similar to 
   	 $$\begin{pmatrix}U^*
   	 SX^{-1}S^*U& U^*(S^{-1})^*\\S^{-1}U& X
   	 \end{pmatrix}=\begin{pmatrix}
   	 |S|X^{-1}|S|& |S|^{-1}\\|S|^{-1}& X
   	 \end{pmatrix}.$$
   	 As $P:=\frac{1}{2}\begin{pmatrix}
   	 I_r\\ I_r
   	 \end{pmatrix}$ is a partial isometry, 
   	 
   	  	\begin{eqnarray*}
   	  		\lambda_r\begin{pmatrix}
   	  			SX^{-1}S^*& (S^{-1})^*\\S^{-1}& X
   	  		\end{pmatrix}&=&\lambda_r\begin{pmatrix}
   	  		|S|X^{-1}|S|& |S|^{-1}\\|S|^{-1}& X
   	  	\end{pmatrix}\\&\ge& \lambda_r \left(P^*\begin{pmatrix}
   	  	|S|X^{-1}|S|& |S|^{-1}\\|S|^{-1}& X
   	  \end{pmatrix}P\right)\\&=&\lambda_r\left(\frac{X+|S|X^{-1}|S|}{2}+|S|^{-1}\right)\\&\ge& \lambda_r(|S|+|S|^{-1})  \ge 2.  \qquad \hbox{by (\ref{am-gm1})}
   	  	\end{eqnarray*}  
   	  	The required result follows. 
   \end{proof}
   
   Now we are ready to give a simpler proof of  Proposition \ref{p1}. 
   
   ~
   
   \noindent
   {\it Proof. } Consider the factorization  $$\begin{pmatrix}
   A_{11}+X & A_{12}\\ A_{12}^*&  A_{12}^*A_{11}^{-1}A_{12}  
   \end{pmatrix}=\begin{pmatrix}
   A_{11}^{\frac{1}{2}}&X^{\frac{1}{2}}\\ A_{12}^*A_{11}^{-\frac{1}{2}}&  0
   \end{pmatrix}\begin{pmatrix}
   A_{11}^{\frac{1}{2}}&X^{\frac{1}{2}}\\ A_{12}^*A_{11}^{-\frac{1}{2}}&  0
   \end{pmatrix}^*.$$
   Clearly, $\begin{pmatrix}
   	A_{11}+X & A_{12}\\ A_{12}^*&  A_{12}^*A_{11}^{-1}A_{12}  
   \end{pmatrix}$ is unitarily similar to  \begin{eqnarray*}
   \begin{pmatrix}
   	A_{11}^{\frac{1}{2}}&X^{\frac{1}{2}}\\ A_{12}^*A_{11}^{-\frac{1}{2}}&  0
   \end{pmatrix}^*\begin{pmatrix}
   A_{11}^{\frac{1}{2}}&X^{\frac{1}{2}}\\ A_{12}^*A_{11}^{-\frac{1}{2}}&  0
\end{pmatrix}&=&\begin{pmatrix}
A_{11}+A_{11}^{-\frac{1}{2}}A_{12}A_{12}^*A_{11}^{-\frac{1}{2}}& A_{11}^{\frac{1}{2}}X^{\frac{1}{2}}\\ X_{11}^{\frac{1}{2}}A^{\frac{1}{2}}&  X
\end{pmatrix}\\&=&\begin{pmatrix}
 A_{11}^{-\frac{1}{2}}X^{-2}A_{11}^{-\frac{1}{2}}& A_{11}^{\frac{1}{2}}X^{\frac{1}{2}}\\ X_{11}^{\frac{1}{2}}A^{\frac{1}{2}}&  X
	\end{pmatrix}.
   \end{eqnarray*} 
   Now setting $S=A_{11}^{-\frac{1}{2}}X^{-\frac{1}{2}}$ in Lemma \ref{lem1}  yields the desired result. \qed
   
   \section{A conjecture}
 A weighted version of  (\ref{bk1}) is known. That is, if $A, B$ are $n\times n$ positive  semidefinite matrices, then   for any $t\in [0, 1]$ and   $j=1, \ldots, n$ 
 \begin{eqnarray}\label{ando} && \lambda_j((1-t)A+tB)\ge  \sqrt{\lambda_j(A^{2(1-t)}B^{2t})}=\sigma_j(A^{1-t}B^{t}).
 \end{eqnarray}   Inequality  (\ref{ando}) is due to Ando \cite{Ando95}. With \ref{ando}), it is not hard to present a weighted version of (\ref{bk2}).
 
   \begin{proposition}\label{p4}  If $A, B$ are $n\times n$ positive  semidefinite matrices, then   for any $t\in [0, 1]$ and   $j=1, \ldots, n$ 
   \begin{eqnarray}\label{lin} \lambda_j((1-t)A+tB)\ge    \lambda_j(A^{1-t}B^{t}). 
   		\end{eqnarray}
     \end{proposition}
  \begin{proof} By (\ref{ando}) and the matrix convexity of the square function,
  	    \begin{eqnarray*} \lambda_j(A^{1-t}B^{t})&=&\sigma_j^2(A^{(1-t)/2}B^{t/2})\\&\le&\lambda_j((1-t)A^{1/2}+tB^{1/2})^2 \\ &\le& \lambda_j((1-t)A+tB). 
  	    \end{eqnarray*}
  \end{proof}
We conclude the paper with the following conjecture \begin{conj}
	  If $A, B$ are $n\times n$ positive definite semidefinite matrices, then for any $t\in [0, 1]$
 	\begin{eqnarray*}
    	\lambda_j((1-t)A+tB)\ge   \sqrt{\sigma_j(A^{2(1-t)}B^{2t})}, \qquad  j=1, \ldots, n.
 	\end{eqnarray*} 
 \end{conj}
    The present method of proof does not seem to lead to a solution of this conjecture.  
    
 \section*{Acknowledgement} The author thanks some helpful conversations with T. Ando and P. van den Driessche.

\end{document}